\documentclass[10pt,reqno]{amsart}
\usepackage{amsmath,amscd,graphicx}
\usepackage{url}

\theoremstyle{plain}
   \newtheorem{theorem}{Theorem}
   \newtheorem{proposition}[theorem]{Proposition}
   \newtheorem{lemma}[theorem]{Lemma}

\theoremstyle{definition}
   
   \newtheorem{question}{Question}
\newtheorem{example}{Example} 

\theoremstyle{remark}
 \newtheorem{remark}{Remark}

\author[P.~Br\"and\'en]{Petter Br\"and\'en}\thanks{Research supported in part by the G\"oran Gustafsson Foundation. A preliminary version containing some of the results of this paper appeared in the conference proceedings of EuroComb 2009.}    
       \address{Department of Mathematics, 
Stockholm University, 
SE-106 91 Stockholm, Sweden}
\email{pbranden@math.su.se}

\keywords{$M$-convex, jump system, matroid, half-plane property, tropicalization, Puiseux series, Tarski's principle, hive, Horn's conjecture}
\subjclass[2000]{90C27, 30C15, 05B35, 15A42}

\newcommand{\NN}{\mathbb{N}}
\newcommand{\ML}{(M${}_{\tiny \text{loc}}$)}

\newcommand{\zz}{\mathbf{z}}
\newcommand{\xx}{\mathbf{x}}
\newcommand{\yy}{\mathbf{y}}

\newcommand{\JJ}{\mathcal{J}}
\newcommand{\FF}{\mathcal{F}}

\newcommand{\HH}{\mathcal{H}}
\newcommand{\BB}{\mathcal{B}}
\newcommand{\MM}{\mathcal{M}}
\newcommand{\PP}{\mathcal{P}}
\newcommand{\A}{\mathcal{A}}

\newcommand{\ZZ}{\mathbb{Z}}
\newcommand{\QQ}{\mathbb{Q}}
\newcommand{\RR}{\mathbb{R}}
\newcommand{\CC}{\mathbb{C}}

\newcommand{\PR}{\mathbb{R}\{t\}}
\newcommand{\PC}{\mathbb{C}\{t\}}

\renewcommand{\Im}{{\rm Im}}
\renewcommand{\Re}{{\rm Re}}

\def\newop#1{\expandafter\def\csname #1\endcsname{\mathop{\rm
#1}\nolimits}}

\newop{diag}
\newop{trop}
\newop{Arg}
\newop{supp}
\newop{per}
\newop{JP}
\newop{Sym}
\newop{St}
\newop{glb}
\newop{Gr}
\newop{Dr}

\begin{document}
\title[Discrete concavity and the half-plane property]
{Discrete concavity and  the half-plane property}
\begin{abstract}
Murota \emph{et al.} have recently developed a theory of discrete convex analysis which concerns $M$-convex functions on jump systems. We introduce here a family of  
$M$-concave functions arising naturally from polynomials  (over a field of generalized Puiseux series) with prescribed non-vanishing properties. This family contains several of the most studied $M$-concave functions in the literature. In the language of tropical geometry we study the tropicalization of the space of polynomials with the half-plane property, and show that it is strictly contained in the space of $M$-concave functions. We also provide a short proof of Speyer's ``hive theorem'' which he used to give a new proof of Horn's conjecture on eigenvalues of sums of Hermitian matrices. 
\end{abstract}
\maketitle

\section{Introduction and Main Results}
Murota and others have recently developed a theory of discrete convex analysis as a framework to solve combinatorial optimization problems using ideas from continuous optimization, see \cite{KMT,MurotaBook,Murota06,MT}. This theory concerns $M$-convex  functions on discrete structures known as jump systems. The work of Choe \emph{et al}. \cite{COSW} and the author \cite{Br} reveal a somewhat surprising relationship between jump systems and supports of multivariate complex polynomials with prescribed non-vanishing properties. The main purpose of this paper is to further study this correspondence and in particular to show that $M$-concave functions arise as valuations of multivariate polynomials over a field of generalized Puiseux series with prescribed non-vanishing properties, see Theorems \ref{Main1} and \ref{Main2}. Similar techniques and ideas are present in tropical geometry.  In particular in 
\cite{Sp} where a correspondence between Vinnikov curves over a field of Puiseux series and discrete concave functions known as hives was used to give an alternative proof of Horn's conjecture on eigenvalues of sums of Hermitian matrices. 
Our results show that the tropicalization of the space of polynomials with the half-plane property is strictly contained in the space of $M$-concave functions. 
In Section \ref{HIVES} we give a short proof, based on a result of Hardy and Hutchinson,  of  Speyer's ``hive theorem''. We also prove that a natural extension of Speyer's theorem to higher dimensions is false. \\[-1ex]

Jump systems were introduced by Bouchet and Cunningham \cite{BC} as a generalization of matroids. 
Let $\alpha,\beta \in \ZZ^n$ and  $|\alpha|=\sum_{i=1}^n|\alpha_i|$. A 
{\em step from $\alpha$  to $\beta$} is 
 an $s \in \ZZ^n$ such that $|s|=1$ and  $|\alpha+s-\beta|=|\alpha-\beta|-1$. 
If $s$ is a step from $\alpha$ to $\beta$ we write $\alpha \stackrel s {\rightarrow} \beta$. A set $\JJ \subseteq \ZZ^n$ is called a {\em jump system} if it respects the following axiom.  
\begin{itemize}
\item[(J1):] If 
$\alpha,\beta \in \JJ$, $\alpha \stackrel s {\rightarrow} \beta$ and  $\alpha+s \notin \JJ$, then there is 
a step $t$ such that $\alpha+s \stackrel t {\rightarrow} \beta$ and $\alpha+s+t \in \JJ$. 
\end{itemize}
Jump systems for which $\JJ \subseteq \{0,1\}^n$ are known as $\Delta$-{\em matroids}, and $\Delta$-matroids for which $|\alpha|=|\beta|$ for all $\alpha, \beta \in \JJ$ coincide with sets of bases of matroids.

A {\em constant parity set} is a set  $\A \subseteq \ZZ^n$ for which $|\alpha|-|\beta|$ is even for all 
$\alpha, \beta \in \A$. Geelen proved that for  constant parity sets the following axiom is equivalent to (J1), see \cite{Murota06}. 
\begin{itemize}
\item[(J2):] If 
$\alpha,\beta \in \JJ$ and $\alpha \stackrel s {\rightarrow} \beta$, then  there is 
a step $t$ such that $\alpha+s \stackrel t {\rightarrow} \beta$, $\alpha+s+t \in \JJ$ and $\beta-s-t \in \JJ$. 
\end{itemize}
Let $\JJ \subseteq \ZZ^n$. A function $f : \JJ \rightarrow \RR$ is $M$-{\em concave} if it respects the next axiom. 
\begin{itemize} 
\item[(M):] If 
$\alpha,\beta \in \JJ$ and $\alpha \stackrel s {\rightarrow} \beta$, then  there is 
a step $t$ such that $\alpha+s \stackrel t {\rightarrow} \beta$, $\alpha+s+t \in \JJ$, $\beta-s-t \in \JJ$ and 
$
f(\alpha)+f(\beta) \leq f(\alpha+s+t)+f(\beta-s-t).
$
The set $\JJ$ is the \emph{support} of $f$. 
\end{itemize}
This concept generalizes that of {\em valuated matroids} \cite{DW}, which are $M$-concave functions with support contained in $\{ \alpha \in \{0,1\}^n : \alpha_1+\cdots + \alpha_n=r\}$ for some $r$. Note that if $f : \JJ \rightarrow \RR$ satisfies (M) then $\JJ$ is a constant parity jump system. Algorithms for maximizing $M$-concave functions on constant parity jump systems have recently been developed in \cite{MT}.

Choe, Oxley, Sokal and Wagner \cite{COSW} initiated the study of combinatorial properties of {\em polynomials with the half-plane property} (HPP-polynomials). Let $H \subset \CC$ be an open half-plane with boundary containing the origin. A multivariate polynomial with complex 
coefficients is  $H$-{\em stable} if it is nonzero whenever all the variables are in $H$. Moreover  if $P$ 
is $H$-stable for some $H$, then $P$ is said to have the {\em half-plane property}. Such polynomials have an intriguing combinatorial structure. Let $P = \sum_{\alpha \in \NN^n} a(\alpha) \zz^\alpha$ be a polynomial in $\CC[\zz]$, where $\zz=(z_1,\ldots, z_n)$ and $\zz^\alpha= z_1^{\alpha_1}\cdots z_n^{\alpha_n}$. The {\em support} of 
$P$ is  $\supp(P)= \{ \alpha \in \NN^n : a(\alpha)\neq 0\}$. A polynomial $P \in K[z_1,\ldots, z_n]$, where $K$ is a field,   is called {\em multiaffine} if $\supp(P) \subseteq \{0,1\}^n$, i.e., if  it can be written as 
$$
P(\zz)= \sum_{S \subseteq \{1,\ldots, n\}} a(S)\prod_{j \in S}z_j,  
$$
where $a(S) \in K$ for all $S \subseteq \{1,\ldots, n\}$.  
A polynomial is {\em homogeneous} if $|\alpha|=|\beta|$ for all $\alpha$, $\beta$ in its support. 
\begin{theorem}[Choe \emph{et al.}, \cite{COSW}]\label{matsup} Let $P \in \CC[\zz]$ be a homogeneous and multiaffine polynomial with the half-plane property. Then $\supp(P)$ is the set of bases of a matroid. 
\end{theorem}
For arbitrary multivariate complex HPP-polynomials  Theorem \ref{matsup} generalizes naturally. 
\begin{theorem}[Br\"and\'en, \cite{Br}]\label{jumpy}
 If $P \in \CC[\zz]$ has the half-plane property then $\supp(P)$ is a jump system.  
\end{theorem}
\begin{remark}
Let $H_0$ be the open upper half-plane. A univariate polynomial with real coefficients is $H_0$-stable if and only if all its zeros are real. Moreover, a multivariate polynomial $P$ with real coefficients is $H_0$-stable if and only if all its zeros along any line with positive slope are real i.e., if all zeros of the polynomial $s \mapsto P(\xx +s\yy)$ are real for all $\xx \in \RR^n$ and $\yy \in (0,\infty)^n$. Hence $H_0$-stability is a natural generalization of real-rootedness.
\end{remark}
%
%

In order to see how $M$-concave functions arise from HPP-polynomials we need to enlarge the field and consider HPP-polynomials over a field with a valuation.  The real field, $\PR$,  of  {\em (generalized) Puiseux series} consists of formal series of the form 
$$
x(t)=\sum_{-k \in A} a_{k}t^{k}
$$
where $A \subset \RR$ is well ordered and $a_k \in \RR$ for all $-k \in A$. The complex field  
of generalized Puiseux series is  $$\PC=\{z=x+iy=\Re(z)+i\Im(z): x,y \in \PR\}.$$ Define the {\em valuation}  $\nu : \PC \rightarrow \RR\cup\{-\infty\}$ to be the map which takes a Puiseux series to its leading exponent, where by convention $\nu(0)=-\infty$. The reason for not choosing the common field of Puiseux series is that we want the valuation to have real values as opposed to rational values. 

A real generalized Puiseux series, $x$, is {\em positive} ($x>0$) if its leading coefficient is positive.  Let $\theta \in \RR$ and  
$H_\theta=\{ z \in \PC : \Im(e^{i\theta}z) >0 \}$ be a half-plane.  A polynomial 
$P \in \PC[\zz]$ is $H_\theta$-{\em stable} if $P \neq 0$ whenever all variables are in $H_\theta$, and it has the half-plane property if it is $H_\theta$-stable for some $\theta \in \RR$.  

The field $\PC$ is algebraically closed and $\PR$ is real closed, see \cite{Ray}. Theorems known to hold for $\RR$ or $\CC$ are typically translated to concern $\PR$ or $\PC$ via Tarski's Principle, see \cite{Sp,Sw} and the references therein. 
\begin{theorem}[Tarski's Principle]
Let $S$ be an elementary statement in the theory of real closed fields. If $S$ is true for one real closed field then $S$ is true in all real closed fields. 
\end{theorem}
The \emph{tropicalization}, $\trop(P)$,  of a polynomial  $P= \sum_{\alpha \in \NN^n}a_\alpha(t)\zz^\alpha \in \PC[\zz]$ is the map 
$\trop(P) : \supp(P) \rightarrow \RR$ defined by $\trop(P)(\alpha)= \nu(a_\alpha(t))$.

We may now state our first main result. 
\begin{theorem}\label{Main1}
Let $P= \sum_{\alpha \in \NN^n}a_\alpha(t)\zz^\alpha \in \PC[\zz]$ and suppose that $\supp(P)$ has constant parity.  If $P$ has the half-plane property then $\trop(P)$ is an $M$-concave function. 
\end{theorem}
\begin{remark}
An important special case of Theorem \ref{Main1} is when we restrict to the class of homogeneous multiaffine 
polynomials, $P(\zz)= \sum_{B \subseteq \{1,\ldots, n\}} a_B(t) \zz^B \in \PC[\zz]$. Then Theorem \ref{Main1} says that the function $\trop(P)$ is a valuated matroid whenever $P$ is a HPP-polynomial.
\end{remark}

Within the class of constant parity jump systems there are those of {\em constant sum}, i.e., $|\alpha|=|\beta|$ for all $\alpha, \beta \in \JJ$. Such jump systems are known to coincide with the set of integer points of 
{\em integral base polyhedra}, see \cite{MurotaBook}.  If $\alpha = (\alpha_1,\ldots, \alpha_j, \ldots, \alpha_n) \in \RR^n$ let $\pi_j(\alpha)= (\alpha_1,\ldots, \alpha_{j-1},\alpha_{j+1}, \ldots, \alpha_n)$.  The {\em projection} of a set $\A \in \ZZ^n$ along a coordinate $j$ is 
$
\pi_j(\A)=\{ \pi_j(\alpha) : \alpha \in \A\}. 
$
The sets that are projections of constant sum jump systems are known to coincide with the set of 
integer points of {\em generalized integral polymatroids}. Such jump systems can be characterized as sets $\JJ \subseteq \ZZ^n$ satisfying the next axiom, see \cite{MurotaBook}. 
\begin{itemize}
\item[(J${}^\natural$):] If 
$\alpha,\beta \in \JJ$ and $\alpha \stackrel s {\rightarrow} \beta$ then  \\
(i) $\alpha +s \in \JJ$ and $\beta -s \in \JJ$, or \\
(ii) there is a step $t$, $\alpha+s \stackrel t {\rightarrow} \beta$ such that   $\alpha+s+t \in \JJ$ and $\beta-s-t \in \JJ$. 
\end{itemize}
Let $\JJ \subseteq \ZZ^n$. A function $f: \JJ \rightarrow \RR$ is $M^\natural$-{\em concave} if it respects the next axiom. 
\begin{itemize}
\item[(M${}^\natural$):] If 
$\alpha,\beta \in \JJ$ and $\alpha \stackrel s {\rightarrow} \beta$ then  \\
(i) $\alpha +s \in \JJ$, $\beta -s \in \JJ$ and 
$
f(\alpha) + f(\beta) \leq f(\alpha+s)+f(\beta-s)$, or \\
(ii) there is a step $t$, $\alpha+s \stackrel t {\rightarrow} \beta$ such that   $\alpha+s+t \in \JJ$ and $\beta-s-t \in \JJ$ and 
$
f(\alpha)+f(\beta) \leq f(\alpha+s+t)+f(\beta-s-t).
$
\end{itemize}
\begin{theorem}\label{Main2}
Let $P= \sum_{\alpha \in \NN^n}a_\alpha(t)\zz^\alpha \in \PR[\zz]$, with $a_\alpha(t) \geq 0$ for all $\alpha \in \NN^n$.  If $P$ is $H_0$-stable then $\supp(P)$ is the set of integer points of an integral generalized  polymatroid and  $\trop(P)$ is an $M^\natural$-concave function. 
\end{theorem}
See Section \ref{Examples} for concrete examples of $M$- and $M^\natural$-concave functions arising from Theorems \ref{Main1} and \ref{Main2}.

We end this section by discussing the case $n=1$ of Theorem \ref{Main2} and postpone the proofs of Theorems \ref{Main1} and \ref{Main2} to the next section.  Let 
$P(z)=\sum_{k=0}^n a_k z^k \in \RR[z]$. Then $P$ is $H_0$-stable if and only if all zeros of $P$ are real. Newton's inequalities then say that 
\begin{equation}\label{newton}
\frac {a_k^2}{{\binom n k}^2}\geq \frac {a_{k-1}}{{\binom n {k-1}}}\frac {a_{k+1}}{{\binom n {k+1}}}, \quad  \mbox{ for all } 1 \leq k \leq n-1.
\end{equation}
Hence if $P(z)=\sum_{k=0}^n a_k(t) z^k \in \PR[z]$ is $H_0$-stable and has nonnegative coefficients then \eqref{newton} holds by Tarski's principle, and consequently 
\begin{equation}\label{valnewton}
2\nu\left(a_k(t)\right) \geq \nu \left(a_{k-1}(t)\right)+ \nu\left(a_{k+1}(t)\right), \quad  \mbox{ for all } 1 \leq k \leq n-1.
\end{equation}
Since $P$ has nonnegative coefficients $\supp(P)$ forms an interval and then \eqref{valnewton} is seen to be equivalent to $M^\natural$-concavity of $f(k) = \nu\left(a_k(t)\right)$. There is also a partial converse 
to Newton's inequalities due to Hardy \cite{Ha} and Hutchinson \cite{Hu}. Let $[M,N]= \{M, M+1,\ldots, N\}$.
\begin{theorem}[Hutchinson, \cite{Hu}]\label{Hthm}
Suppose $P(z)=\sum_{k=M}^N a_k z^k \in \RR[z]$ where $a_k >0$ for all $k \in [M,N]$. 
If 
\begin{equation}\label{lc4}
a_k^2 \geq 4a_{k-1}a_{k+1}, \quad  \mbox{ for all } M<k<N,
\end{equation}
then all zeros of $P$ are real. 
Moreover if \eqref{lc4} holds with strict inequalities then $P$ has no multiple zeros except possibly $z=0$.
\end{theorem}
Hardy \cite{Ha} proved Theorem \ref{Hthm} with the 4 replaced by a 9. 
\begin{remark}\label{onevar}
It follows from Theorem \ref{Hthm}
and Tarski's principle that if $f : [M,N] \rightarrow \RR$ is $M^\natural$-concave,  then the polynomial 
$$
P(z)= \sum_{k=M}^N 4^{-\binom k 2} t^{f(k)}z^k
$$
is $H_0$-stable over $\PR$. Also, if $P(z)= \sum_{k=M}^Na_k(t) z^k  \in \PR[z]$ where $a_k(t) >0$ for all $k \in [M,N]$ and    
$$
2\nu\left(a_k(t)\right) > \nu \left(a_{k-1}(t)\right)+ \nu\left(a_{k+1}(t)\right), \quad \mbox{ for all } k \in [M+1,N-1], 
$$
then $P(z)$ is $H_0$-stable. 
\end{remark}
\section{Proofs of Theorem \ref{Main1} and Theorem \ref{Main2}}
We start by discussing polarization procedures for jump systems, $M$-concave functions and HPP-polynomials. 
  
If $\A \subset \NN^n$ is a finite set and $j\in [1,n]$ let $\kappa_j= \max\{ \alpha_j : (\alpha_1,\ldots, \alpha_n) \in \A\}$ and  $V_\kappa = \{ v_{ij} : 1\leq i \leq n \mbox{ and } 0 \leq j \leq \kappa_j\}$ where all the $v_{ij}$'s  are distinct. Define a projection  $\Pi_\kappa^\downarrow : \{0,1\}^{V_\kappa} \rightarrow \NN^n$  by 
$$
\Pi_\kappa^\downarrow(\sigma)= \left(\sum_{j=0}^{\kappa_1}\sigma(v_{1j}), \ldots, \sum_{j=0}^{\kappa_n}\sigma(v_{nj})\right).
$$

The {\em polarization} of $\A$ is 
$
\PP(\A)= \{ \sigma \in \{0,1\}^{V_\kappa}: \Pi_\kappa^\downarrow(\sigma) \in \A \}. 
$
Similarly if $f: \A \rightarrow \RR$ define the {\em polarization}, $f^{\uparrow} : \PP(\A) \rightarrow \RR$, of $f$ by $f^{\uparrow}(\sigma)=f\left(\Pi_\kappa^\downarrow(\sigma)\right)$. 
\begin{proposition}\label{polsystem}
Let $\A \subset \NN^n$ be a finite set and $f : \A \rightarrow \RR$. Then 
\begin{itemize} 
\item[(1)] $\A$ is a jump system if and only if $\PP(\A)$ is a $\Delta$-matroid;
\item[(2)] If $\A$ has constant parity then $f$ is $M$-concave if and only if $f^\uparrow$ is $M$-concave. 
\end{itemize}
\end{proposition}
\begin{proof}
This is almost immediate from the definitions. For a proof of (1) see \cite{KS}. 
\end{proof}
 Let $P \in \PC[z_1,\ldots, z_n]$ be a polynomial of degree $d_i$ in the variable $z_i$ for all $1\leq i \leq n$. The 
{\em polarization}, $\PP(P)$, is the unique polynomial in the variables 
$\{ z_{ij} : 1 \leq i \leq n \mbox{ and } 1 \leq j \leq d_i\}$ satisfying 
\begin{enumerate}
\item $\PP(P)$ is multiaffine; 
\item $\PP(P)$ is symmetric in the variables $z_{i1}, \ldots, z_{id_i}$ for all $1 \leq i \leq n$;
\item If we make the change of variables   $z_{ij}=z_i$ for all $i,j$ in $\PP(P)$ we recover $P$.
\end{enumerate}
Note that $\supp(\PP(P))=\PP(\supp(P))$.
\begin{proposition}\label{polarizeit}
Let $P \in \PC[z_1,\ldots, z_n]$ and let $H$ be a half-plane in $\CC$. Then 
$P$ is $H$-stable if and only if $\PP(P)$ is $H$-stable. 
\end{proposition}
\begin{proof}
For the corresponding statement over $\CC$, see \cite{COSW} or \cite[Proposition 2.4]{BB}. This can be translated to a statement concerning $\RR$ (or $\PR$) by identifying $\CC$ with $\RR \times \RR$ (or $\PC$ with $\PR \times \PR$), and considering $P$ to be a function from $\RR^n\times \RR^n$ to $\RR \times \RR$.  Hence, the theorem also holds for $\PC$ by Tarski's Principle. 
\end{proof}
Murota \cite{Murota06} proved that if $\JJ$ is a constant parity jump system, then a function $f : \JJ \rightarrow \RR$ is $M$-{concave} if and only if it respects the following local axiom. 
\begin{itemize}
\item[(M${}_{\tiny \text{loc}}$):] If 
$\alpha,\beta \in \JJ$ and $|\alpha - \beta|=4$, then there are steps $s,t$ such that $\alpha \stackrel s {\rightarrow} \beta$, $\alpha+s \stackrel t {\rightarrow} \beta$, $\alpha+s+t \in \JJ$, $\beta-s-t \in \JJ$ and 
$$
f(\alpha)+f(\beta) \leq f(\alpha+s+t)+f(\beta-s-t).
$$ 
\end{itemize}

Real multiaffine polynomials with the half-plane property with respect to the upper half-plane are characterized by inequalities (compare with \ML). Proposition \ref{SR} was originally formulated for $\RR$ but holds also for $\PR$ by Tarski's Principle.  
\begin{proposition}[Br\"and\'en, \cite{Br}]\label{SR}
Let $P \in \PR[z_1,\ldots, z_n]$ be multiaffine and let $H_0$ be the open upper half-plane. 
Then $P$ is $H_0$-stable if and only if 
$$
\frac{\partial P}{\partial z_i}(\xx) \frac{\partial P}{\partial z_j}(\xx) - \frac{\partial^2P}{\partial z_i\partial z_j}(\xx) P(\xx) \geq 0
$$
for all $i,j \in \{1,\ldots, n\}$ and $\xx \in \PR^n$. 
\end{proposition}
We now have all tools to proceed with the proof of Theorem \ref{Main1}.
\begin{proof}[Proof of Theorem \ref{Main1}]
Let $P= \sum_{\gamma \in \NN^n}a_\gamma(t)\zz^\gamma \in \PC[z_1,\ldots, z_n]$ and suppose that $P$ has the half-plane property and that $\JJ=\supp(P)$ has constant parity.  $\JJ$ is a constant parity jump system by Theorem \ref{jumpy} and Tarski's Principle. By Propositions \ref{polsystem} and \ref{polarizeit} we may assume that 
$P$ is multiaffine and that $\JJ$ is a $\Delta$-matroid.  To prove the validity of \ML, assume that  $\alpha, \beta \in \JJ$ with $|\alpha-\beta|=4$.  By a rotation of the variables ($P(e^{i\theta}z_1,\ldots, e^{i\theta}z_n)$ for some $\theta \in \RR$) we may assume that $P$ has the half-plane property with respect to the right half-plane. But then, by \cite[Theorem 6.2]{COSW}, we may assume that all nonzero coefficients are positive. Since $\Re(z) >0$ if and only if  $\Re(z^{-1})>0$ the operation 
$$
P(z_1,\ldots, z_j, \ldots, z_n) \mapsto z_jP(z_1, \ldots, z_j^{-1}, \ldots, z_n)
$$
preserves the half-plane property with respect to the right half-plane (and the constant parity property). By performing such operations for 
the indices satisfying $\alpha_j > \beta_j$ we may in fact assume that $\alpha_j \leq \beta_j$ for all $j$. 
Suppose that $\alpha_i =1$ and $\beta_j=0$. By Hurwitz' theorem\footnote{See \cite{COSW} for an appropriate multivariate version.} the polynomials
$$
\lim_{\lambda \rightarrow \infty} \lambda^{-1}P(z_1, \ldots, z_{i-1}, \lambda, z_{i+1}, \ldots, z_n), \quad \lim_{\lambda \rightarrow \infty}  P(z_1, \ldots, z_{j-1}, \lambda^{-1}, z_{j+1}, \ldots, z_n)
$$
are right half-plane stable. If necessary, by performing a few such operation we end up with (by reindexing the variables and indices) a right half-plane stable polynomial 
$Q(z_1,z_2,z_3,z_4)$, and the vectors we want to check the validity of \ML{ } are  $\alpha = (0,0,0,0), \beta=(1,1,1,1) \in \supp(Q)$. Since all coefficients of $Q$ are nonnegative  the polynomial 
\begin{equation*}
\begin{split}
G(z_1,z_2, z_3, z_4)&= Q(-iz_1,-iz_2,-iz_3,-iz_4) \\ &= 
b_{0000}(t)- b_{0011}(t)z_3z_4-b_{0101}(t)z_2z_4-b_{0110}(t)z_2z_3 \\ &- b_{1001}(t)z_1z_4-b_{1010}(t)z_1z_3-b_{1100}(t)z_1z_2 + b_{1111}(t)z_1z_2z_3z_4
\end{split}
\end{equation*}
is upper half-plane stable with real coefficients. So is the polynomial 
\begin{equation*}
\begin{split}
F(z_1,z_2, z_3, z_4)&= G(z_1,z_2,t^{\lambda}z_3,z_4) \\ &= 
a_{0000}(t)- a_{0011}(t)z_3z_4-a_{0101}(t)z_2z_4-a_{0110}(t)z_2z_3 \\ &- a_{1001}(t)z_1z_4-a_{1010}(t)z_1z_3-a_{1100}(t)z_1z_2 + a_{1111}(t)z_1z_2z_3z_4, 
\end{split}
\end{equation*}
where $\lambda$ is any real number. 

We are now in a position to apply Proposition \ref{SR}. 
\begin{equation*}\begin{split}
&\frac{\partial F}{\partial z_1}(0,0,1,x) \frac{\partial F}{\partial z_2}(0,0,1,x) - \frac{\partial^2F}{\partial z_1\partial z_2}(0,0,1,x) F(0,0,1,x) =\\ 
&x^2 (a_{1001}a_{0101}+a_{1111}a_{0011})+x(a_{1001}a_{0110}+a_{1010}a_{0101}-a_{0000}a_{1111}-a_{1100}a_{0011}) \\ &+ a_{1010}a_{0110}+a_{0000}a_{1100} =x^2A+xB+C\geq 0.
\end{split}
\end{equation*}
Hence the discriminant, $\Delta=B^2-4AC$, of the above quadratic in $x$ is nonpositive (by Theorem \ref{jumpy} and Tarski's principle). In order to get a contradiction assume 
$$
\nu(b_{0000})+\nu(b_{1111}) > \max\Big( \nu(b_{1001})+\nu(b_{0110}),  \nu(b_{1010})+\nu(b_{0101}), \nu(b_{1100})+\nu(b_{0011})\Big).
$$
Then 
\begin{equation}\label{ina}
\begin{split}
&\nu(a_{0000})+\nu(a_{1111}) >  \\
&\max\Big( \nu(a_{1001})+\nu(a_{0110}),  \nu(a_{1010})+\nu(a_{0101}), \nu(a_{1100})+\nu(a_{0011})\Big)
\end{split}
\end{equation}
for all $\lambda \in \RR$. We shall see that for some $\lambda$ the discriminant $\Delta$ will be positive. By \eqref{ina}, $\nu(B^2)=2\nu(a_{0000})+2\nu(a_{1111})=:W(\lambda)$. Note that $W(\lambda)=W(0)+2\lambda$. Also,  by \eqref{ina}, 
$W(\lambda)$ is greater than the valuation of each term in the expansion of $4AC$ except possibly for  
$$
U(\lambda):=\nu(a_{1010})+\nu(a_{0110})+\nu(a_{1111})+\nu(a_{0011})=U(0)+4\lambda
$$
and 
$$
V(\lambda):=\nu(a_{0000})+\nu(a_{1100})+\nu(a_{1001})+\nu(a_{0101})=V(0). 
$$
Hence it remains to prove that for some $\lambda_0$ 
\begin{equation}\label{uvw}
\max( U(\lambda_0), V(\lambda_0) ) < W(\lambda_0), 
\end{equation}
because then, for $\lambda_0$,  $\nu(B^2)>\nu(4AC)$ and thus $\Delta>0$. 
Suppose that $U(0) \neq -\infty$ and $V(0) \neq -\infty$. 
For $\lambda$ small enough we have $U(\lambda)< W(\lambda)< V(\lambda)$ and for $\lambda$ large enough we have $V(\lambda)< W(\lambda)< U(\lambda)$.  It follows that there is a number $\lambda_0$ for which $U(\lambda_0)=V(\lambda_0)$. However, \eqref{ina} implies 
$
U(\lambda)+V(\lambda) < 2W(\lambda)
$
 for all $\lambda$, so $U(\lambda_0)=V(\lambda_0)< W(\lambda_0)$.  The case when $U(0) = -\infty$ or $V(0) = -\infty$ follows similarly. 
\end{proof}
If $\JJ \subset \NN^n$ is a finite set with $\max\{|\alpha| : \alpha \in \JJ\}=r$ and $f: \JJ \rightarrow \RR$ let 
$$\tilde{\JJ} = \left\{ (\alpha_1,\ldots, \alpha_{n+1}) \in \NN^{n+1}: (\alpha_1,\ldots, \alpha_{n})\in \JJ, \alpha_{n+1}=r-\sum_{j=1}^n\alpha_j\right\}, 
$$
and let $\tilde{f} : \tilde{\JJ} \rightarrow \RR$ be defined by $\tilde{f}(\alpha_1,\ldots, \alpha_{n+1})=f(\alpha_1,\ldots, \alpha_{n})$. 
Proofs of the next two propositions can be found in \cite{MurotaBook} and \cite{BBL}. 
\begin{proposition}\label{hom1}
Let $\JJ \subset \NN^n$ be a finite set and $f : \JJ \rightarrow \RR$. Then $f$ is 
$M^\natural$-concave if and only if $\tilde{f}$ is $M$-concave.
\end{proposition}
\begin{proposition}\label{hom2}
Suppose that $P \in \PR[z_1,\ldots, z_n]$ has degree $r$ and that all coefficients in $P$ are nonnegative. Let $\tilde{P}(z_1,\ldots, z_{n+1})=z_{n+1}^rP(z_1/z_r, \ldots, z_n/z_r)$. Then 
$P$ is $H_0$-stable if and only if $\tilde{P}$ is $H_0$-stable.
\end{proposition}
The proof of Theorem \ref{Main2} is now immediate. 
\begin{proof}[Proof of Theorem \ref{Main2}]
Combine Theorem \ref{Main1} and Propositions \ref{hom1} and \ref{hom2}.
\end{proof}
\section{Examples of Tropical HPP-polynomials}\label{Examples}
To illustrate Theorems \ref{Main1} and \ref{Main2} we provide here examples that show that known $M$-concave functions are tropicalizations of  HPP-polynomials. 
\begin{example}
If $a$ is a positive number then $1+az_1z_2$ has the the half-plane property with respect to the open right half-plane (since the product of two complex numbers with positive real part is never a negative real number). Let $w \in \RR$. By Tarski's principle the polynomial $1+t^{w}z_1z_2$ is a HPP-polynomial over $\PR$ (with respect to the open right half-plane). 
Let $G=(V,E)$ be a graph with no loops, $w : E \rightarrow \RR$  and define 
$$
P_G(\zz)= \prod_{ij=e\in E}(1+t^{w(e)}z_iz_j)=\sum_{\alpha \in \NN^n}a_\alpha(t)\zz^\alpha.
$$
The support of $P_G$ is the set of degree sequences of subgraphs of $G$ and is by Theorem \ref{jumpy} a constant parity jump system. The tropicalization of $P_G$ is given by 
$$
\trop(P_G)(\alpha)= \nu(a_\alpha(t))= \max \left\{ \sum_{e \in H}w(e): H \subseteq E, (V,H) \mbox{ has degree sequence } \alpha \right\}.
$$
This function is $M$-concave by Theorem \ref{Main1} (as proved by Murota \cite{Murota06}). 
\end{example}
The next example shows that the classical ``maximum weighted matching problem" is a special case of maximizing an $M$-concave function arising from a HPP-polynomial. 
\begin{example}
Let $G=(V,E)$ be a finite graph with no loops  and let $\JJ \subseteq 2^V$ be 
the set of vertices of partial matchings of $G$, i.e., $S \in \JJ$ if there is a perfect matching of the subgraph of $G$ induced by $S$. Write $\FF \leadsto S$ to indicate that the set $\FF \subseteq E$ is the set of edges of a perfect matching of the subgraph induced by $S$. Let $w : E \rightarrow \RR$ and define 
$f : \JJ \rightarrow \RR$ by 
$$
f(S)= \max\left\{ \sum_{e \in \FF}w(e): \FF \leadsto S\right\}.
$$
Murota \cite{Murota97} proved that $f$ is $M$-concave. Clearly $f = \trop(P_G)$ where 
$$
P_G(\zz)= \sum_{\FF} t^{\sum_{e \in \FF}w(e)}\prod_{ij \in \FF}z_iz_j, 
$$
and where the sum is over all partial matchings of $G$. That $P_G$ is a HPP-polynomial follows immediately from Tarski's Principle and the multivariate Heilmann-Lieb theorem, see \cite[Theorem 10.1]{COSW}. 
\end{example}
\begin{example} Let $G=(V,E)$, $V=\{1,\ldots, n\}$  be a graph with no loops and let $w : E \rightarrow \RR$, and $c : E \rightarrow \NN$. Then 
$$
P_G(\zz)= \prod_{ij=e\in E}(1+t^{w(e)}z_iz_j)^{c(e)} = \sum_{\alpha \in \NN^n} a_\alpha(t)\zz^\alpha
$$
has the half-plane property over $\PR$. By Theorem \ref{Main1} the function $f : \supp(P_G) \rightarrow \RR$ defined by $$
f(\alpha)= \nu(a_\alpha(t))= \max\left\{ \sum_{e \in E}w(e)b(e) : b(e) \in \NN \cap [0, c(e)], 
\alpha = \left(\sum_{e \in \delta(j)}b(e)\right)_{j=1}^n\right\},
$$  
where $\delta(j)$ denotes the set of edges incident to $j$, is $M$-concave. This function is studied in \cite{Murota06,MT}.
\end{example}
\begin{example}
Let $A_1(t), \ldots, A_n(t)$ be positive semi-definite $d \times d$ matrices over $\PC$. Then 
the polynomial 
$$
P(\zz) = \det\Big(z_1A_1(t)+ \cdots + z_nA_n(t)\Big)
$$
has the half-plane property over $\PR$, see \cite{Br, Sp}. Hence, $\trop(P)$ is $M$-concave. 
\end{example}

\begin{example}
If $A=A(t)$ is an $r \times n$ matrix over $\PC$ let 
$$
P_A(\zz)=\sum_{|S|=r} \det(A(t)[S]) \overline{\det(A(t)[S])} \prod_{j \in S}z_j, 
$$
where $A(t)[S]$ is the $r \times r$ minor with columns indexed by $S \subseteq \{1, \ldots, n\}$.  Then  
$P_A(\zz)$ has the half-plane property over $\PR$, see e.g., \cite{Br}. Hence 
the function $\trop(P_A)$ is $M$-concave, i.e., a valuated matroid. This is true also for fields other than $\CC$, although our method won't work. Let $\binom {[n]} r = \{ \alpha \in \{0,1\}^n : \alpha_1 + \cdots+ \alpha_n =r\}$. The space of all functions $\trop(P_A) : \binom {[n]} r \rightarrow \RR \cup \{-\infty\}$ where $A$ is an $r \times n$ matrix over $\PC$ coincides with the \emph{tropical Grassmannian}, $\Gr(r,n)$,  as studied in 
\cite{HJJS, SpSt}. In \cite{HJJS} the \emph{Dressian}, $\Dr(r,n)$, is defined as the space of $M$-concave functions (valuated matroids) 
$f : \binom {[n]} r \rightarrow \RR \cup \{-\infty\}$. Let $\HH(r,n)$ be the space of all HPP-polynomials with support contained in $\binom {[n]} r$. We have the inclusions 
$$
\Gr(r,n) \subseteq \trop(\HH(r,n)) \subseteq \Dr(r,n).
$$
We shall see that for $r=4$ and $n=8$ the inclusions are strict. A matroid $\MM$ on $[n]$ has the \emph{weak half-plane property} if there is a HPP-polynomial $P$ with support equal to the set of bases of $\MM$.  There are several matroids on $8$ elements of rank $4$ that fail to have the weak half-plane property, see \cite{BG}. For such a matroid $\MM$, let $f_\MM : \binom {[8]} 4 \rightarrow \RR \cup \{-\infty\}$ be defined by 
$$
f_\MM(S) = \begin{cases}
1 &\mbox{if $S$ is basis,} \\
-\infty &\mbox{otherwise.}
\end{cases}
 $$
 It follows that $f_\MM \in \Dr(4,8) \setminus \trop(\HH(4,8))$. 
 
 The V\'amos matroid $V_8$ is not representable over any field, see \cite{Ox}. However $V_8$ has the weak half-plane property, see \cite{WW}. It follows that 
 $f_{V_8} \in \trop(\HH(4,8))\setminus \Gr(4,8)$. 
\end{example}
\begin{example}
Let $A(t)$ be a skew symmetric $n \times n$ matrix over $\PR$. Then 
$$
\sum_{S \subseteq \{1,\ldots, n\}} \det(A[S]) \prod_{j \in S}z_j, 
$$
where $A[S]$ is the principle minor indexed by $S$, has the half-plane property over $\PR$, see \cite[Corollary 4.3]{Br}. Hence $f(S) = \nu(\det(A[S]))$ is $M$-concave, i.e., a valuated $\Delta$-matroid.  This is known to be true over any field, see \cite{DW91}. 
\end{example}

\section{Hives and Horn's Problem}\label{HIVES}
Let $\Delta_n=\{ \alpha \in \NN^3 : \alpha_1+\alpha_2 + \alpha_3=n\}$. $M$-concave functions on 
$\Delta_n$ are better known as {\em hives} and were used in the resolution of Horn's problem on eigenvalues of sums of Hermitian matrices, and in the proof of
the saturation conjecture, see \cite{Bu,KT,Sp}.  If we depict $\Delta_n$ as in Fig. \ref{hh},
\begin{figure}[htp]
 \centering
 \includegraphics[height=1.3in]{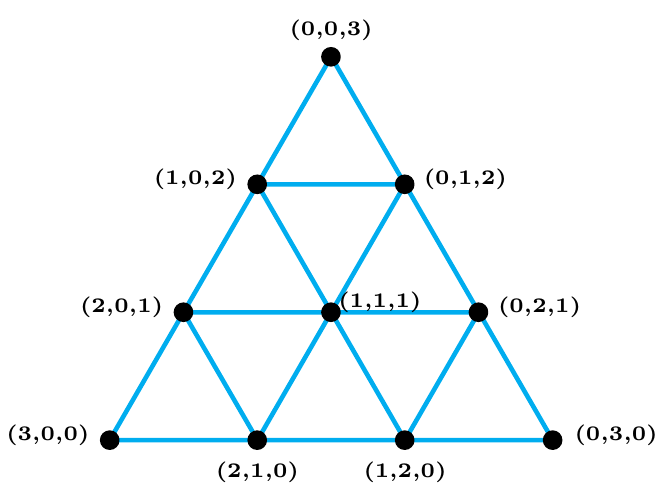}
\caption{\label{hh} $\Delta_3$.}
\end{figure}
then a function $h : \Delta_n \rightarrow \RR$ is called a hive if the {\em rhombus inequalities} in Fig. \ref{tru} are satisfied by $h$. It is clear that $M$-concave functions are hives and hives are easily seen to satisfy 
\ML, so a function on $\Delta_n$ is a hive if and only if it is $M$-concave. 
\begin{figure}[htp]
 \centering
 \includegraphics[height=0.9in]{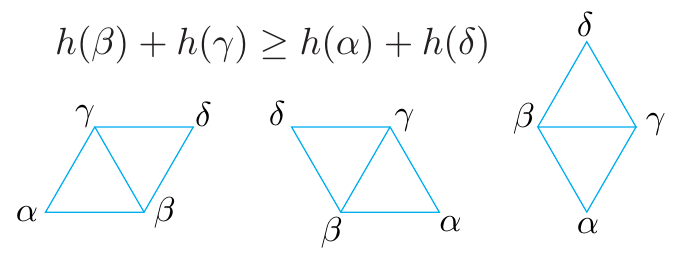}
\caption{\label{tru} Rhombus inequalities}
\end{figure}

The recently established Lax conjecture, see \cite{HV,LPR},  implies a characterization of HPP-polynomials over $\RR$ with support 
$\Delta_n$ as polynomials of the form 
\begin{equation}\label{Vinnikov}
\pm \det(xA +yB+zC), 
\end{equation}
where $A,B,C$ are positive definite symmetric (or Hermitian) $n \times n$ matrices. This makes the connection between Horn's problem and HPP-polynomials. A hive $h : \Delta_n \rightarrow \RR$ is {\em strict} if (M) (or equivalently all rhombus inequalities) hold with strict inequalities. 
\begin{theorem}[Speyer, \cite{Sp}]\label{Spe}
Suppose $P=\sum_{\alpha \in \Delta_n}a_\alpha(t)x^{\alpha_1}y^{\alpha_2}z^{\alpha_3} \in \PR[x,y,z]$ has positive coefficients and let $h= \trop(P)$. If $P$ is a HPP-polynomial then $h$ is a hive, and if $h$ is a strict hive then $P$ is a HPP-polynomial. 

Moreover, if $h : \Delta_n \rightarrow \QQ$ is a hive then there is a HPP-polynomial $P \in \PR[x,y,z]$ with $h= \trop(P)$.
\end{theorem}
Note that Theorem \ref{Spe} implies that the tropicalization of the space HPP-polynomials with support $\Delta_n$ coincides with the space of all hives with support $\Delta_n$. 

Horn conjectured a characterization of  all possible triples $\alpha^1, \alpha^2, \alpha^3 \in \RR^n$ such that 
for $j=1,2,3$, $\alpha^j$ are the eigenvalues of a Hermitian $n \times n$ matrix $A_j$ and  $A_1+A_2=A_3$. Horn's conjecture was first proved by Klyachko \cite{Kl} and Knutson-Tao \cite{KT}, see \cite{Fu} for a survey. 
 Speyer \cite{Sp} used his theorem and \eqref{Vinnikov} to give a new proof of Horn's conjecture.  The  proof uses Viro's patchworking method. We give here a short proof of Theorem \ref{Spe} based on Remark \ref{onevar} and the following simple lemma. 
\begin{lemma}\label{basic}
Suppose that $P(x,y,z) \in \RR[x,y,z]$ is a homogeneous polynomial with nonnegative coefficients such that $P(0,0,1)P(0,1,0)>0$ and the univariate polynomials $x \mapsto P(x, 1, \lambda)$, $y \mapsto P(1,y,\lambda)$, $z \mapsto P(1,\lambda,z)$, $y \mapsto P(1,y,0)$ and $z \mapsto P(1,0,z)$ have only real zeros, for all $\lambda >0$. Then $P$ has the half-plane property i.e., it is $H_0$-stable. 
\end{lemma}
\begin{proof}
Since $P$ is homogeneous with nonnegative coefficients it has the half-plane property if and only if 
$P(1,y,z)$ is upper half-plane stable, see e.g., \cite[Theorem 4.5]{BBL}. Suppose that $P$ satisfies the hypothesis in the  lemma  and that $P(1,y_0,z_0)=0$ for some $y_0, z_0 \in H_0$. If 
$\Arg(y_0)=\Arg(z_0)$ then there is a $\lambda>0$ such that $P(y_0^{-1},1,\lambda)=0$ which is a contradiction. By symmetry we may assume that $\Arg(y_0) > \Arg(z_0)$.  Let $y(s)= s +(1-s)y_0$ where $0 \leq s \leq 1$. Since $P(0,0,1) \neq 0$ the zeros of the polynomials $z \mapsto  P(1,y(s),z)$ where $s \in [0,1]$ are bounded. Let $[0,1] \ni s \mapsto z(s)$ be a continuous curve\footnote{Such a curve always exists. In fact if $F(s,y)= \sum_{k=0}^M Q_k(s) z^k \in \CC[s,z]$ with $Q_M(s) \neq 0$ for all $s \in [0,1]$ we can find a continuous parametrization $[0,1] \ni s \mapsto (z_1(s), \ldots, z_M(s)) \in \CC^M$  of the zeros of $z \mapsto F(s,z)$.  First we may assume $F$ is irreducible.  Then the discriminant, $\Delta(s)$, of $F$ with $s$ fixed  is not identically zero as a polynomial in $s$ (otherwise $F$ would have a multiple factor by Gauss' lemma). Thus there is a finite number of $s \in [0,1]$ where $z \mapsto F(s,z)$ has multiple zeros, and these divide  $[0,1]$ into a finite number of  open intervals. By Hurwitz' theorem on the continuity of zeros (see e.g., \cite{Ma}) we can find a continuous parametrization in each open interval. By continuity of the zeros as a multiset we can reorder the zeros so that the parametrizations glue together to a continuous parametrization for $[0,1]$. }
 in $\CC$ such that $z(0)=z_0$ and $P(1,y(s),z(s))=0$. We have the following possible cases. 
\begin{enumerate}
\item If $z(s) \in H_0$ for all $s \in [0,1]$ then $P(1,1, z(1)) =0$;
\item  If $z(s)> 0$ for some $s \in [0,1]$ then $P(1,y,\lambda)=0$  for some $y \in H_0\cup \{1\}$ and $\lambda>0$; 
\item If $z(s)=0$ for some $s\in [0,1]$, then $P(1,y,0)=0$ for some $y \in H_0\cup \{1\}$;  
\item If there is an $s_0 \in (0,1]$ for which $z(s) \in H_0$ for $0 \leq s<s_0$ and 
$z(s_0) <0$, let 
$
\delta(s) = \Arg(y(s))-\Arg(z(s)).
$
Then $\delta(0)>0$ and $\delta(s_0)= \Arg(y(s_0))-\pi <0$ so by continuity $\delta(s_1)=0$ for some $0<s_1<s_0$. But then 
$P(1, y, \lambda y)=y^dP(y^{-1},1,\lambda)=0$ for  $y=y(s_1) \in H_0$ and $\lambda>0$.
\end{enumerate}
Since all cases above lead to contradictions $P$ must be a HPP-polynomial.
\end{proof}
\begin{proof}[Proof of Theorem \ref{Spe}] One direction is just a special case of Theorem \ref{Main1} so assume that 
$P$ is as in Theorem \ref{Spe} with $h$  a strict hive.  Lemma \ref{basic} also holds for $\PR$ by Tarski's Principle.  Since 
$2h(\alpha_1,\alpha_2,0)> h(\alpha_1-1,\alpha_2+1,0)+ h(\alpha_1+1,\alpha_2-1,0)$ Remark \ref{onevar} verifies that $P(1,y,0)$ is $H_0$-stable. Let $P(1,\lambda,z)= 
\sum_{k =0}^na_k(t)z^k$. Then $\nu(a_k(t))= \max\{h(\alpha)+\alpha_2\nu(\lambda) : \alpha \in \Delta_n \mbox{ and } \alpha_3=k\}$. Since also $h(\alpha)+\alpha_2\nu(\lambda)$ is a strict hive, (M) implies 
$2\nu(a_k(t))> \nu(a_{k-1}(t))+\nu(a_{k+1}(t))$ for all $1 \leq k<n$  which by Remark \ref{onevar} proves that $P(1,\lambda,z)$ is a HPP-polynomial. This proves the theorem by Lemma \ref{basic} and Tarski's principle.
\end{proof}
We may  also derive a quantitative version of Theorem \ref{Spe}. The {\em rhombus quotients} of a 
homogeneous polynomial 
$\sum_{\alpha \in \Delta_n}a_\alpha x^{\alpha_1}y^{\alpha_2}z^{\alpha_3} \in \RR[x,y,z]$ with positive coefficients are the set 
of quotients 
$
{a_\beta a_\gamma}/{a_\alpha a_\delta},
$
where $\alpha, \beta, \gamma, \delta$ form a rhombus as in Fig.~\ref{tru}.
\begin{theorem}\label{quant}
Let $P(x,y,z) = \sum_{\alpha \in \Delta_n}a_\alpha x^{\alpha_1}y^{\alpha_2}z^{\alpha_3}$ be a homogenous polynomial of degree $n$ with positive coefficients. 
\begin{itemize}
\item[(a)] If $P$ is a HPP-polynomial and $\alpha, \beta, \gamma, \delta$ is a rhombus as in Fig. \ref{tru} then 
$$
\frac {a_\beta a_\gamma} {a_\alpha a_\delta} \geq \frac {\ell +1}{2\ell}, 
$$
where $\ell$ is the common coordinate of $\beta$ and $\gamma$; 
\item[(b)] If all rhombus quotients are greater or equal to $2(n-1)$ then $P$ is a HPP-polynomial.
\end{itemize}
\end{theorem}
\begin{proof}
For (a) see \cite{Sp}. 

Let $P$ be as in the statement of the theorem with all rhombus quotients greater or equal to $Q= 2(n-1)$, where $n \geq 2$. Set 
$a_\alpha = Q^{h(\alpha)}$. We want to prove that $P$ satisfies the conditions in Lemma \ref{basic}. We prove that all zeros of $z \mapsto P(1,\lambda,z)=\sum_{k=0}^n a_k z^k$ are real, the other cases follow similarly. Since the polynomial $P(x,\lambda y, z)$ also has all rhombus quotients  greater or equal to $Q$, we may assume that $\lambda=1$. By assumption 
$h(\beta)+h(\gamma) \geq h(\alpha)+h(\delta)+1$ for each rhombus as in Fig.~\ref{tru}. Hence we may write $h$ as $h=h_0+h_1$, where $h_0(i,j,k)= -\binom i 2 -\binom j 2 - \binom k 2$ and $h_1$ is a hive. 
The extension (linearly on all small triangles) of a hive  to the set $\Delta^\RR_n=\{(x,y,z) \in \RR^3 : x,y,z \geq 0 \mbox{ and } x+y+z=n\}$ is concave (see \cite{Bu,MurotaBook}) and we denote this extension by the same symbol. 
Let $D_k=\{ \alpha \in \Delta_n : \alpha_3=k\}$ and let $R: D_{k-1}\times D_{k+1} \rightarrow \NN$ be defined by $R(\alpha, \delta)= 2h_0((\alpha+\delta)/2)-h_0(\alpha)-h_0(\delta)$. Then $R(\alpha, \delta) \geq 1$, and $R(\alpha, \delta) \geq 2$ unless $\alpha$ and $\delta$ are in the same rhombus. Hence
\begin{eqnarray*}
a_{k+1}a_{k-1} &=& \mathop{\sum_{\alpha \in D_{k-1}}}_{\delta \in D_{k+1}}Q^{h(\alpha)+h(\delta)} 
\leq \mathop{\sum_{\alpha \in D_{k-1}}}_{\delta \in D_{k+1}}Q^{2h((\alpha+\delta)/2)-R(\alpha,\delta)}  \\
&\leq& (n-k)\frac {Q^{-1}}{2} \mathop{\sum_{\beta, \gamma \in D_k}}_{\beta \neq \gamma}Q^{h(\beta)+h(\gamma)} + (n-k)Q^{-2}\sum_{\gamma \in D_k}Q^{2h(\gamma)} \\
&\leq& (n-k)\frac {Q^{-1}} 2 \sum_{\beta,\gamma \in D_{k}}Q^{h(\beta)+h(\gamma)}  =  (n-k)\frac {Q^{-1}} 2a_k^2. 
\end{eqnarray*}
The second inequality comes from splitting the previous sum into  two sums, $S_1+S_2$, one where $\kappa:=(\alpha+\delta)/2 \notin D_k$  and the other where $\kappa \in D_k$. If 
$\kappa \notin D_k$ then $\kappa = (\beta +\gamma)/2$ for a unique $\{\beta, \gamma\} \subseteq  D_k$ for which $|\beta-\gamma|=2$.  There are at most 
$n-k$ pairs $\alpha \in D_{k-1}, \delta \in D_{k+1}$ for which $(\alpha+\delta)/2=\kappa$ for a specific $\kappa \notin D_k$. Also, $2h((\alpha+\delta)/2)-R(\alpha,\delta) \leq h(\beta)+h(\gamma)-1$ which explains the first sum in the second row. The second sum $S_2$ is estimated similarly. Hence $a_k^2 \geq 4 a_{k-1}a_{k+1}$ which by Theorem \ref{Hthm} proves that all zeros of $z \mapsto P(1,\lambda,z)$ are real and the theorem follows.
\end{proof}

\section{Higher Dimensional Hives}
Let $\Delta_n^m= \{ \alpha \in \NN^m : \alpha_1 + \cdots + \alpha_m=n\}$. We extend the definition of   a hive to mean an $M$-concave function on $\Delta_n^m$. It is natural to ask if an analog of Theorem \ref{Spe} holds for $\Delta_n^m$ when $m>3$. In particular one may ask if all higher dimensional hives are tropicalizations of HPP-polynomials.
\begin{question}\label{dense}
Suppose that $h : \Delta_n^m \rightarrow \RR$ is $M$-concave. Is there a HPP-polynomial 
$P$ such that 
$\trop(P)=h$. 
\end{question}
We will prove in this section that Question \ref{dense} is not true in general. Suppose that $P = \sum_{\alpha \in \NN^n}a_\alpha(t)\zz^\alpha \in \PR[z]$ has nonnegative coefficients. We say that $P$ is an $M$\emph{-polynomial} if $\supp(P)$ has constant parity and $\trop(P)$ is $M$-concave. 
\begin{proposition}\label{operations}
Let $P \in \PR[z_1,\ldots, z_n]$ be an $M$-polynomial. Then so are (unless identically zero) 
\begin{enumerate}
\item $P(z_1+w_1,z_2, \ldots, z_n)$ where $w_1$ is a new variable; 
\item $P(z_1,z_1, z_3,\ldots, z_n)$; 
\item $P(\xi z_1, z_2, \ldots, z_n)$ whenever $\xi \in \PR$ is nonnegative; 
\item ${\partial P}/{\partial z_1}$. 
\end{enumerate}
\end{proposition}
\begin{proof}
The proposition is a reformulation of the fact that the operations considered in \cite{KMT} preserve $M$-convexity. 
\end{proof}
\begin{remark}\label{H-closed}
The class of $H$-stable polynomials is also closed under the operations in Proposition \ref{operations}, 
see \cite{BB,COSW}. 
\end{remark}

\begin{lemma}\label{dist}
Let $\JJ \subseteq \Delta_n^m$ be a constant sum jump system.
The function 
$d_\JJ : \Delta_n^m \rightarrow \QQ$ defined by 
$$
d_\JJ(\alpha) = -\min\{ |\alpha-\beta| : \beta \in \JJ\}
$$
is $M$-concave and  $\JJ = \{\alpha \in \Delta_n^m :   d_\JJ(\alpha) \geq d_\JJ(\beta) \mbox{ for all } \beta \in \Delta_n^m \}$. 
\end{lemma}
\begin{proof}
By definition the polynomial $P(z_1,\ldots, z_m)= \sum_{\alpha \in \JJ}\zz^\alpha$ is an $M$-polynomial. By Proposition \ref{operations} so is the polynomial 
$$
Q(\zz) = \sum_{\alpha \in \Delta_n^m}b_\alpha(t)\zz^\alpha = P\left(z_1+t^{-1}\sum_{j=2}^m z_j, \ldots, z_m + t^{-1}\sum_{j=1}^{m-1}z_j\right).
$$
Now $2\nu(b_\alpha(t))=d_\JJ(\alpha)$ for all $\alpha \in \Delta_n^m$ and the lemma follows. 
\end{proof} 

\begin{proposition}
Let $\BB_7 \subset \Delta_3^7$ be the set of bases of the Fano matroid $\FF_7$, see \cite{Ox}.  The function $d_{\BB_7}$ fails Question \ref{dense}. 
\end{proposition}
\begin{proof} 
Suppose that there is a  HPP-polynomial  $P=\sum_{\alpha \in \Delta_3^7} a_\alpha(t) \zz^\alpha$ such that $\supp(P)=\Delta_3^7$ and $\trop(P) = d_{\BB_7}$. Let $\delta=\max \{ \nu(a_\alpha(t)) : \alpha \in \supp(P) \}$ and let $\widetilde{a_\alpha}$ be the coefficient of $\delta$ in $a_\alpha(t)$. It follows from Propositions \ref{polarizeit} and \ref{SR} that the polynomial $\widetilde{P}= \sum_{\alpha \in \Delta_3^7} \widetilde{a_\alpha} \zz^\alpha \in \RR[\zz]$  has the half-plane property. However, by construction, $\supp(\widetilde{P})=\BB_7$. Thus $\FF_7$ has the weak half-plane property which contradicts  \cite[Theorem 6.6]{Br}. 
\end{proof}

\noindent 
\textbf{Acknowledgments.} 
I am grateful to the anonymous referees for valuable comments and for spotting a mistake in the proof of Theorem \ref{Main1}.

\end{document}